\documentclass[a4paper,10pt]{amsart}
\usepackage{amssymb}
\usepackage{amsmath}
\usepackage{amscd}

\begin{document}

\theoremstyle{plain}
\newtheorem{theorem}{Theorem}[section]
\newtheorem{proposition}[theorem]{Proposition}
\newtheorem{lemma}[theorem]{Lemma}
\newtheorem{corollary}[theorem]{Corollary}
\newtheorem{conj}[theorem]{Conjecture}

\theoremstyle{definition}
\newtheorem{definition}[theorem]{Definition}
\newtheorem{exam}[theorem]{Example}
\newtheorem{remark}[theorem]{Remark}

\numberwithin{equation}{section}

\title[Einstein metrics on $F_4$]
{Non-naturally reductive Einstein metrics on the compact simple Lie group $F_4$}

\author{Zhiqi Chen}
\address{School of Mathematical Sciences and LPMC \\ Nankai University \\ Tianjin 300071, P.R. China} \email{chenzhiqi@nankai.edu.cn}

\author{Ke Liang}
\address{School of Mathematical Sciences and LPMC \\ Nankai University \\ Tianjin 300071, P.R. China}\email{liangke@nankai.edu.cn}



\begin{abstract}
Based on the representation theory and the study on the involutions
of compact simple Lie groups, we show that $F_4$ admits
non-naturally reductive Einstein metrics.
\end{abstract}

\maketitle


\setcounter{section}{0}
\section{Introduction}
A Riemannian manifold $(M,\langle\ ,\ \rangle)$ is called Einstein
if the Ricci tensor $\mathrm{Ric}$ of the metric $\langle\ ,\
\rangle$ satisfies $\mathrm{Ric}=c\langle\ ,\ \rangle$ for some
constant $c$. For the study on homogeneous Einstein metrics, see
\cite{Bo1,Bo2,Bo3,WZ1}, and the survey \cite{NR1}. 

In particular, there are a lot of important results on the study of
Einstein metrics on Lie groups. It is proved in \cite{Ja1} that any
Einstein solvmanifold, that is, a simply connected solvable Lie
group endowed with a left invariant metric, is standard. According
to a long standing conjecture attributed to Alekseevskii (see
\cite{Be1}), this exhausts the class of noncompact homogeneous
Einstein manifolds. It is well known that Einstein metrics on a
solvable Lie group are unique up to isometry and scaling \cite{He1},
which is in sharp contrast to the compact setting. For a compact Lie
group $G$, the problem is how to give Einstein metrics on it.

In \cite{DZ1}, D'Atri and Ziller gives a large number of
left invariant Einstein metrics, which are naturally reductive. The
problem of finding non-naturally reductive left invariant Einstein
metrics on compact Lie groups seems harder, and in fact this is
stressed in \cite{DZ1}. In \cite{Mo1}, Mori obtains non-naturally
reductive Einstein metrics on the Lie group $SU(n)$ for $n\geq 6$ by
using the method of Riemannian submersions. Based on the discussion
on K\"ahler C-spaces, \cite{AMY1} gives non-naturally reductive
Einstein metrics on compact simple Lie groups $SO(n)$ for $n\geq
11$, $Sp(n)$ for $n\geq 3$, $E_6$, $E_7$, and $E_8$. In \cite{GLP},
there are non-naturally reductive Einstein metrics on low
dimensional Lie groups such as $SU(3)$, $SO(5)$ and $G_2$. 

Up to now, it seems no non-naturally reductive Einstein metric on $F_4$. Based on the representation theory and the study on the
involutions of compact simple Lie groups, we get two new Einstein
metrics on $F_4$, one of which is non-naturally reductive. Thus we prove the following theorem.
\begin{theorem}\label{main}
The compact simple Lie group $F_4$ admits non-naturally reductive Einstein metrics.
\end{theorem}

\section{The previous study on Einstein metrics on compact Lie groups}
Let $(M, g)$ be a Riemannian manifold and $I(M,g)$ the Lie group of
all isometries of $M$. Then $(M,g)$ is said to be $K$-homogeneous if
a Lie subgroup $K$ of $I(M,g)$ acts transitively on $M$. For a
$K$-homogeneous Riemannian manifold $(M,g)$, we fix a point $o\in M$
and identify $M$ with $K/L$ where $L$ is the isotropy subgroup of
$K$ at $o$. Let $k$ be the Lie algebra of $K$ and ${\mathfrak l}$
the subalgebra corresponding to $L$. Take a vector space ${\mathfrak
p}$ complement to ${\mathfrak l}$ in ${\mathfrak k}$ with
$Ad(L){\mathfrak p}\subset {\mathfrak p}$. Then we may identify
${\mathfrak p}$ with $T_o(M)$ in a natural way. We can pull back the
inner product go on $T_o(M)$ to an inner product on ${\mathfrak p}$,
denoted by $\langle\ ,\ \rangle$. For $X\in {\mathfrak k}$ we will
denote by $X_{\mathfrak l}$ ( resp. $X_{\mathfrak p}$) the
${\mathfrak l}$-component (resp. ${\mathfrak p}$-component) of $X$.
A homogeneous Riemannian metric on $M$ is said to be naturally
reductive if there exist $K$ and ${\mathfrak p}$ as above such that
$$\langle [Z,X]_{\mathfrak p} ,Y \rangle + \langle X, [Z, Y ]_{\mathfrak p}\rangle =0, \text{ for any } X,Y,Z\in {\mathfrak p}.$$

In \cite{DZ1}. D'Atri and Ziller have investigated naturally
reductive metrics among the left invariant metrics on compact Lie
groups, and have given a complete classification in the case of
simple Lie groups. Let $G$ be a compact connected semi-simple Lie
group, $H$ a closed subgroup of $G$, and let ${\mathfrak g}$ be the
Lie algebra of $G$ and ${\mathfrak h}$ the subalgebra corresponding
to $H$. We denote by $B$ the negative of the Killing form of
${\mathfrak g}$. Then $B$ is an $Ad(G)$-invariant inner product on
${\mathfrak g}$. Let ${\mathfrak m}$ be a orthogonal complement of
${\mathfrak h}$ with respect to $B$. Then we have $${\mathfrak
g}={\mathfrak h}\oplus {\mathfrak m},\quad Ad(H){\mathfrak m}\subset
{\mathfrak m}.$$ Let ${\mathfrak h}={\mathfrak h}_0\oplus {\mathfrak
h}_1\oplus\cdots\oplus{\mathfrak h}_p$ be the decomposition into
ideals of ${\mathfrak h}$, where ${\mathfrak h}_0$ is the center of
${\mathfrak h}$ and ${\mathfrak h}_i(i=1,\cdots,p)$ are simple
ideals of ${\mathfrak h}$. Let $A_0|_{{\mathfrak h}_0}$ be an
arbitrary metric on ${\mathfrak h}_0$.

\begin{theorem}[\cite{DZ1}] Under the notations above, a left
invariant metric on $G$ of the form
\begin{equation}\label{natural}
\langle\ ,\ \rangle=x\cdot B|_{\mathfrak m}+A_0|_{\mathfrak
h_0}+u_1\cdot B|_{\mathfrak h_1}+\cdots+u_p\cdot B|_{\mathfrak h_p}
\quad (x, u_1, \cdots, u_p \in {\mathbb R}^+) \end{equation} is
naturally reductive with respect to $G\times H$, where $G\times H$
acts on $G$ by $(g, h)y = gyh^{-1}$. Conversely, if a left invariant
metric $\langle\ ,\  \rangle$ on a compact simple Lie group $G$ is
naturally reductive, then there exists a closed subgroup $H$ of $G$
such that the metric $\langle\ ,\ \rangle$ is given by the form
(\ref{natural}).
\end{theorem}

There is a simple formula for the Ricci curvature of a left
invariant metric which was derived in \cite{Sa1}. In \cite{DZ1},
D'Atri and Ziller gave a simple proof.

\begin{lemma}[\cite{DZ1,Sa1}]\label{Ziller}
For any $x,y\in{\mathfrak g}$,
$\mathrm{Ric}(x,y)=-\mathrm{tr}(\nabla_x-\mathrm{ad}
x)(\nabla_y-\mathrm{ad} y)$.
\end{lemma}

Based on the Lemma~\ref{Ziller}, D'Atri and Ziller gave a
sufficient and necessary condition for the left invariant metrics
given in~(\ref{natural}) to be Einstein. By computing the case when
${\mathfrak h}$ acts irreducibly on ${\mathfrak m}$, they get the
following number of Einstein metrics on the exceptional groups:
$$5 \text{ on } G_2,\quad 10 \text{ on } F_4, \quad 14 \text{ on } E_6,\quad 15 \text{ on } E_7, \quad 11 \text{ on } E_8,$$
and the following number of Einstein metrics on the classical
groups:
$$n+1 \text{ on } SU(2n), SU(2n+3), Sp(2n), Sp(2n-1), \quad 3n-2 \text{ on } SO(2n), SO(2n+1),$$
and several more in isolated dimensions. 

Let $\mathfrak g$ be a compact simple Lie algebra, let $\theta$ be
an involution of $\mathfrak g$, and let $${\mathfrak h}=\{x\in
{\mathfrak g}|\theta(x)=x\}, \quad {\mathfrak m}=\{x\in {\mathfrak
g}|\theta(x)=-x\}.$$ Then ${\mathfrak g}={\mathfrak h}\oplus
{\mathfrak m}$, $({\mathfrak h},{\mathfrak m})=0$ and ${\mathfrak
h}$ acts irreducibly on ${\mathfrak m}$. For this case, by the
theory of ansatz, Mujtaba \cite{Mu1} and Pope \cite{Po1}
gave the following number of non-equivalent Einstein metrics on some
classical groups respectively:
$$2n+1 \text{ on } SU(2n), \quad 2n \text{ on } SU(2n+1),\quad 3n-4 \text{ on } SO(2n),\quad 3n-3 \text{ on } SO(2n+1).$$

But all the Einstein metrics mentioned above are naturally
reductive. In order to find non-naturally reductive Einstein
metrics, \cite{AMY1} gives a further discussion based on the form
(\ref{natural}).

Let ${\mathfrak m}={\mathfrak m}_1\oplus\cdots\oplus {\mathfrak
m}_q$ be a decomposition of $\mathfrak m$ into irreducible
$Ad(H)$-modules ${\mathfrak m}_j$ for any $j=1,\cdots,q$, and assume
that the $Ad(H)$-modules ${\mathfrak m}_j$ are mutually
non-equivalent, and that the ideals ${\mathfrak h}_i, i=1,\cdots,p$
of ${\mathfrak h}$ are mutually non-isomorphic. In particular,
assume that $\dim {\mathfrak h}_0\leq 1$. We consider the following
left invariant metric on $G$ which is $Ad(H)$-invariant:
$$\langle\ ,\ \rangle=u_0\cdot B|_{\mathfrak h_0}+u_1\cdot B|_{\mathfrak h_1}+\cdots+u_p\cdot B|_{\mathfrak h_p}+x_1\cdot B|_{\mathfrak m_1}+\cdots+x_q\cdot B|_{\mathfrak
m_q},$$ where $u_0, u_1,\cdots, u_p, x_1,\cdots, x_q\in {\mathbb
R}^+$. In order to compute the Ricci tensor of the left invariant
metric $\langle\ ,\ \rangle$ on $G$, write the decomposition
${\mathfrak g}={\mathfrak h}_0\oplus {\mathfrak h}_1
\oplus\cdots\oplus {\mathfrak h}_p\oplus {\mathfrak m}_1 \oplus
\cdots \oplus {\mathfrak m}_q$ as ${\mathfrak g}={\mathfrak
w}_0\oplus {\mathfrak w}_1\oplus\cdots\oplus{\mathfrak w}_p\oplus
{\mathfrak w}_{p+1}\oplus \cdots\oplus{\mathfrak w}_{p+q}$.

Note that the space of left invariant symmetric covariant 2-tensors
on $G$ which are $Ad(H)$-invariant is given by
\begin{equation}\label{non-natural}
\{v_0\cdot B|_{\mathfrak w_0}+ v_1\cdot B|_{\mathfrak w_1}+\cdots +
v_{p+q}\cdot B|_{\mathfrak w_{p+q}} | v_0, v_1, \cdots, v_{p+q}\in
{\mathbb R}\}. \end{equation}

In particular, the Ricci tensor $r$ of a left invariant Riemannian
metric $\langle\ ,\ \rangle$ on $G$ is a left invariant symmetric
covariant 2-tensor on $G$ which is $Ad(H)$-invariant and thus $r$ is
of the form (\ref{non-natural}). Let ${e_\alpha}$ be a
$B$-orthonormal basis adapted to the decomposition of ${\mathfrak
g}$, i.e., $e_\alpha\in {\mathfrak w}_i$ for some $i$, and
$\alpha<\beta$ if $i<j$ (with $e_\alpha\in {\mathfrak w}_i$ and
$e_\beta\in {\mathfrak w}_j$). We set
$A^\gamma_{\alpha\beta}=B([e_\alpha, e_\beta], e_\gamma)$ so that
$[e_\alpha, e_\beta]=\sum_\gamma A^\gamma_{\alpha\beta}e_\gamma$,
and set $\left[\begin{array}{c} k \\ ij \end{array}\right]=\sum
(A^\gamma_{\alpha\beta})^2$, where the sum is taken over all indices
$\alpha, \beta,\gamma$ with $e_\alpha\in {\mathfrak w}_i$,
$e_\beta\in {\mathfrak w}_j$ and $e_\gamma\in {\mathfrak w}_k$. Then
$\left[\begin{array}{c} k \\ ij \end{array} \right]$ is independent
of the $B$-orthonormal bases chosen for ${\mathfrak w}_i, {\mathfrak
w}_j ,{\mathfrak w}_k$, and $\left[\begin{array}{c} k \\ ij
\end{array} \right]=\left[\begin{array}{c} k \\ ji \end{array}
\right]=\left[\begin{array}{c} j \\ ki \end{array}\right]$. Write a
metric on $G$ of the form (\ref{non-natural}) as
$$\langle\ ,\ \rangle=y_0\cdot B|_{\mathfrak w_0}+y_1\cdot B|_{\mathfrak w_1}+\cdots+y_p\cdot B|_{\mathfrak
w_p}+ y_{p+1}\cdot B|_{\mathfrak w_{p+1}}+\cdots+ y_{p+q}\cdot
B|_{\mathfrak w_{p+q}}$$ where $y_0, y_1,\cdots, y_{p+q}\in {\mathbb
R}^+$.

\begin{lemma}[\cite{AMY1}]\label{lemma3.3}
Let $d_k=\dim {\mathfrak w_k}.$ The components $r_0, r_1,\cdots,
r_{p+q}$ of the Ricci tensor $r$ of the above metric $\langle\ ,\
\rangle$ on $G$ are given by
$$r_k=\frac{1}{2y_k}+\frac{1}{4d_k}\sum_{j,i}\frac{y_k}{y_jy_i}\left[\begin{array}{c} k \\
ji\end{array}\right]-\frac{1}{2d_k}\sum_{j,i}\frac{y_j}{y_ky_i}\left[\begin{array}{c}
j \\ ki \end{array}\right]$$ for any $k=0,1,\cdots, p+q)$, where the
sum is taken over all $i,j=0,1,\cdots, p+q$. Moreover, for each $k$
we have $\sum_{i,j}\left[\begin{array}{c} j \\ ki
\end{array}\right]=d_k.$
\end{lemma}

Based on the above lemma and a further discussion for K\"ahler
C-spaces,

\begin{theorem}[\cite{AMY1}]
The compact simple Lie groups $SO(n)$ for $n\geq 11$, $Sp(n)$ for
$n\geq 3$, $E_6$, $E_7$, and $E_8$ admit non-naturally reductive
Einstein metrics.
\end{theorem}

\section{A Ricci curvature formula associated with an involution pair}
The Lie algebra $\mathfrak h$ of the compact subgroup $H$ in the
work of Mujtaba and Pope is the fixed points of an
involution of the Lie algebra ${\mathfrak g}$ of the compact simple
Lie group $G$.

The Lie subalgebra $\mathfrak h$ in \cite{AMY1} is also associated
to the involution of the Lie algebra $\mathfrak g$. In fact, Let
$\Pi=\{\alpha_1, \cdots, \alpha_l\}$ be a fundamental system of
$\Delta$, where $\Delta$ is the root system of $\mathfrak g$
corresponding to a maximal abelian subalgebra $\mathfrak t$. Let
$\phi=\sum_{j=1}^lm_j\alpha_j$ be the highest root. First choose a
simple root $\alpha_i$ such that $m_i=2$. Then there exists an
element $h_1$ of $\mathfrak t$ such that
$$\alpha_i(h_1)=\pi\sqrt{-1}, \quad \alpha_j(h_1)=0 \text{ if }
j\not=i.$$ Then $\rho=e^{\mathrm{ad}h_1}$ is an involution of
$\mathfrak g$. Let ${\mathfrak k}=\{x\in {\mathfrak g}|\rho(x)=x\}$
and ${\mathfrak m}=\{x\in {\mathfrak g}|\rho(x)=-x\}$. Then
$\Pi-\{\alpha_i\}\cup\{-\phi\}$ is a fundamental system of
$\mathfrak k$. There exists an element $h_2$ of $\mathfrak t$ such
that $$-\phi(h_2)=\pi\sqrt{-1},\quad \alpha_j(h_2)=0 \text{ if }
j\not=i.$$ Then $\varphi=e^{\mathrm{ad}h_2}$ is an involution of
$\mathfrak k$, and
$${\mathfrak h}=\{x\in {\mathfrak k}|\varphi(x)=x\}.$$
Furthermore, $\varphi$ can be extended to an automorphism of
$\mathfrak g$ satisfying
$$\alpha_i(h_2)=-\frac{\pi}{2}\sqrt{-1}, \quad \alpha_j(h_2)=0 \text{ if }
j\not=i.$$ It means that $\varphi^4=id$ as a linear map on
${\mathfrak g}$. By a result of Yan \cite{Ya1}, if an
involution $\varphi$ of $\mathfrak k$ can be extended to be an
automorphism of $\mathfrak g$, also denoted by $\varphi$, then
$$\varphi^4=id \text{ or } \varphi^2=id.$$

\begin{remark}
The cases given in \cite{AMY1} are just a part of the cases for
$\varphi^4=id$.
\end{remark}

\begin{remark}
For the latter, it corresponds to the commuting involutions on
compact simple Lie algebras, which leads to an classification of the
simple locally symmetric pairs. For the details, see
\cite{CL,CH,Ya1}.
\end{remark}

Let $\rho,\varphi$ be involutions of $\mathfrak g$ satisfying $\rho\varphi=\varphi\rho$. Then we have a decomposition
$${\mathfrak g}={\mathfrak h_1}\oplus{\mathfrak h_2}\oplus{\mathfrak h_3}\oplus{\mathfrak h_4},$$
corresponding to $\rho,\varphi$, where \begin{eqnarray*}
  && {\mathfrak h_1}=\{x\in {\mathfrak
g}|\rho(x)=x, \varphi(x)=x\}, \quad {\mathfrak h_2}=\{x\in
{\mathfrak
g}|\rho(x)=x, \varphi(x)=-x\},\\
  &&{\mathfrak h_3}=\{x\in {\mathfrak
g}|\rho(x)=-x, \varphi(x)=x\},\quad {\mathfrak h_4}=\{x\in
{\mathfrak g}|\rho(x)=-x, \varphi(x)=-x\}.
\end{eqnarray*}
Let $B$ be an associative inner product on $\mathfrak g$. It is
easy to check that $B({\mathfrak h_i}, {\mathfrak h_j})=0 \text{ if
} j\not=i$ and
\begin{eqnarray*}
  && [{\mathfrak h_1}, {\mathfrak h_i}]\subset {\mathfrak h_i},
  \quad i=1,2,3,4;\quad [{\mathfrak h_2}, {\mathfrak h_2}]\subset {\mathfrak
  h_1};
  \quad [{\mathfrak h_2}, {\mathfrak h_3}]\subset {\mathfrak h_4};
  \\
  && [{\mathfrak h_2}, {\mathfrak h_4}]\subset {\mathfrak h_3};
  \quad [{\mathfrak h_3}, {\mathfrak h_3}]\subset {\mathfrak h_1};
  \quad  [{\mathfrak h_3}, {\mathfrak h_4}]\subset {\mathfrak h_2};
  \quad [{\mathfrak h_4}, {\mathfrak h_4}]\subset {\mathfrak
  h_1}.
\end{eqnarray*}
We consider the following left invariant metric on
$G$:
\begin{equation}\label{involutionpair}
\langle\ ,\ \rangle=u_1\cdot B|_{\mathfrak h_1}+u_2\cdot
B|_{\mathfrak h_2}+u_3\cdot B|_{\mathfrak h_3}+u_4\cdot
B|_{\mathfrak h_4},\end{equation} where $u_1, \cdots, u_4\in
{\mathbb R}^+$.

\begin{lemma}\label{natural2}
If a left invariant metric $\langle\ ,\ \rangle$ of the form
$(\ref{involutionpair})$ on $G$ is naturally reductive with respect
to $G\times L$ for some closed subgroup $L$ of $G$, then one of the
following holds:
\begin{enumerate}
  \item $u_2=u_3=u_4$, or
  \item Assume that $\{i_2,i_3,i_4\}=\{2,3,4\}$. Then $u_1=u_{i_2}$ and $u_{i_3}=u_{i_4}$, or
  \item $u_1=u_2=u_3=u_4$, which implies that
  $(\ref{involutionpair})$ is a bi-invariant metric.
\end{enumerate}
Conversely,
\begin{enumerate}
  \item If $u_2=u_3=u_4$, then the metric $\langle\ ,\ \rangle$ is given by
  $u_1\cdot B|_{\mathfrak h_1}+u_2\cdot B|_{{\mathfrak h}_2\oplus {\mathfrak h}_3\oplus {\mathfrak h}_4}$, which is
naturally reductive with respect to $G\times H_1$, where the
Lie algebra of $H_1$ is ${\mathfrak h}_1$.
  \item If $u_1=u_{i_2}$ and $u_{i_3}=u_{i_4}$, then the metric $\langle\ ,\ \rangle$ is given by
  $u_1\cdot B|_{{\mathfrak h}_1\oplus {\mathfrak h}_{i_2}}+u_{i_3}\cdot B|_{{\mathfrak h}_{i_3}\oplus {\mathfrak h}_{i_4}}$,
  which is naturally reductive with respect to $G\times K$, where the Lie algebra of $K$ is ${\mathfrak h}_1\oplus {\mathfrak h}_{i_2}$.
\end{enumerate}
\end{lemma}
\begin{proof}
The lemma follows form Theorem~\ref{natural} and the fact that
only $\mathfrak h_1$ and ${\mathfrak h}_1\oplus {\mathfrak h}_i$ for
some i are subalgebras of $\mathfrak g$.
\end{proof}
\begin{lemma}
The Levi-Civita connection corresponding to the above metric is
\[ \nabla_xy=\left\{ \begin{aligned}
  &\frac{1}{2}[x,y],  \quad\forall x,y\in {\mathfrak h}_i; \\
&\frac{2u_i-u_1}{2u_i}[x,y],  \quad\forall x\in {\mathfrak h}_1, y\in {\mathfrak h}_i, i>1; \\
&\frac{u_3+u_4-u_2}{2u_4}[x,y],  \quad\forall x\in {\mathfrak h}_2, y\in
{\mathfrak h}_3; \\
&\frac{u_3+u_4-u_2}{2u_3}[x,y],  \quad\forall x\in {\mathfrak h}_2, y\in
{\mathfrak h}_4; \\
&\frac{u_2+u_4-u_3}{2u_2}[x,y],  \quad\forall x\in {\mathfrak h}_3, y\in
{\mathfrak h}_4;
\end{aligned}\right.\]
and $[x,y]=\nabla_xy-\nabla_yx$.
\end{lemma}
\begin{proof}
The lemma follows from direct computations.
\end{proof}

Let $\{e_1,\cdots,e_n\}$ be an orthogonal basis of ${\mathfrak g}$
with respect to $B$, where $\{e_1,\cdots,e_{i_1}\}$ is the basis of
${\mathfrak h}_1$, $\{e_{i_1}+1,\cdots,e_{i_2}\}$ is the basis of
${\mathfrak h}_2$, $\{e_{i_2}+1,\cdots,e_{i_3}\}$ is the basis of
${\mathfrak h}_3$ and $\{e_{i_3}+1,\cdots,e_{n}\}$ is the basis of
${\mathfrak h}_4$. Let $i_0=0$ and $i_4=n$. Let
$\{e_1^*,\cdots,e_n^*\}$ denote the dual basis of
$\{e_1,\cdots,e_n\}$ corresponding to the left invariant metric
$(\ref{involutionpair})$. Then
\begin{equation}e_k^*=\frac{1}{u_{j+1}}e_k, \text{ for any } k\in
{i_j+1,\cdots,i_{j+1}}.\end{equation}
\begin{proposition}\label{Ricci}
The Ricci curvature is given as follows.
\begin{enumerate}
  \item $\mathrm{Ric}(x,y)=0$, for any $x\in {\mathfrak h}_i, y\in {\mathfrak
  h}_j$ if $i\not=j;$
  \item For any $x,y\in {\mathfrak h}_1,$
$\mathrm{Ric}(x,y)=\sum_{j=0}^3-\frac{u_1}{4u_{j+1}^2}\langle\sum_{k={i_j+1}}^{{i_{j+1}}}\mathrm{ad}^2e_k(x),y\rangle;$
  \item For any $x,y\in {\mathfrak h}_2,$
  \begin{eqnarray*}
  &&\mathrm{Ric}(x,y)\\&=&-\frac{u_1}{4u_{2}^2}\langle\sum_{k=1}^{i_1}\mathrm{ad}^2e_k(x),y\rangle-\frac{4u_2-3u_1}{4u_2^2}\langle\sum_{k=i_1+1}^{i_2}\mathrm{ad}^2e_k(x),y\rangle
  \\
  &&+\frac{2u_4(u_4-u_2-u_3)+(u_4-u_2+u_3)(u_4+u_2-u_3)}{4u_2u_3u_4}\langle\sum_{k={i_2+1}}^{i_3}\mathrm{ad}^2e_k(x),y\rangle\\
  &&+\frac{2u_3(u_3-u_2-u_4)+(u_3-u_2+u_4)(u_3+u_2-u_4)}{4u_2u_3u_4}\langle\sum_{k={i_3+1}}^{i_4}\mathrm{ad}^2e_k(x),y\rangle;
                                         \end{eqnarray*}
  \item For any $x,y\in {\mathfrak h}_3,$
  \begin{eqnarray*}
  &&\mathrm{Ric}(x,y)\\
  &=&-\frac{u_1}{4u_{3}^2}\langle\sum_{k=1}^{i_1}\mathrm{ad}^2e_k(x),y\rangle-\frac{4u_3-3u_1}{4u_3^2}\langle\sum_{k={i_2+1}}^{i_3}\mathrm{ad}^2e_k(x),y\rangle\\
                   &&+
                   \frac{2u_4(u_4-u_2-u_3)+(u_4-u_2+u_3)(u_4+u_2-u_3)}{4u_2u_3u_4}\langle\sum_{k={i_1+1}}^{i_2}\mathrm{ad}^2e_k(x),y\rangle
                   \\
                   &&+\frac{2u_2(u_2-u_3-u_4)+(u_2-u_3+u_4)(u_2+u_3-u_4)}{4u_2u_3u_4}\langle\sum_{k={i_3+1}}^{i_4}\mathrm{ad}^2e_k(x),y\rangle;\end{eqnarray*}
  \item For any $x,y\in {\mathfrak h}_4,$ \begin{eqnarray*}
  &&\mathrm{Ric}(x,y)\\&=&-\frac{u_1}{4u_{4}^2}\langle\sum_{k=1}^{i_1}\mathrm{ad}^2e_k(x),y\rangle-\frac{4u_4-3u_1}{4u_4^2}\langle\sum_{k={i_3+1}}^{i_4}\mathrm{ad}^2e_k(x),y\rangle\\
                   &&+
                   \frac{2u_3(u_3-u_2-u_4)+(u_3-u_2+u_4)(u_3+u_2-u_4)}{4u_2u_3u_4}\langle\sum_{k={i_1+1}}^{i_2}\mathrm{ad}^2e_k(x),y\rangle
                   \\
                   &&+\frac{2u_2(u_2-u_3-u_4)+(u_2-u_3+u_4)(u_2+u_3-u_4)}{4u_2u_3u_4}\langle\sum_{k={i_2+1}}^{i_3}\mathrm{ad}^2e_k(x),y\rangle;\end{eqnarray*}
\end{enumerate}
\end{proposition}
\begin{proof}
Let notation be as above. Then \begin{eqnarray*}
\mathrm{Ric}(x,y)&=&
\sum_{j=0}^3\sum_{k={i_j+1}}^{i_{j+1}}\langle R(e_k,x)y, e_k^*\rangle \\
&=& \sum_{j=0}^3\sum_{k={i_j+1}}^{i_{j+1}}\frac{1}{u_{j+1}}\langle
\nabla_{e_k}\nabla_xy-\nabla_x\nabla_{e_k}y-\nabla_{[e_k,x]}y,e_k\rangle
\end{eqnarray*}
If $x\in {\mathfrak h}_i, j\in {\mathfrak h}_j$, where $i\not=j$,
then for any $z\in {\mathfrak h}_k$,
$$\nabla_{z}\nabla_xy\not\in{\mathfrak h}_k, \quad\nabla_x\nabla_{z}y\not\in{\mathfrak h}_k,\quad\nabla_{[z,x]}y\not\in{\mathfrak h}_k.$$
It follows that $\mathrm{Ric}(x,y)=0$ for any $x\in {\mathfrak h}_i,
y\in {\mathfrak h}_j$ if $i\not=j$.

Assume that $x,y\in {\mathfrak h}_1$. For any $e\in {\mathfrak
h}_i$, we have
\begin{eqnarray*}
\langle
\nabla_e\nabla_xy,e\rangle&=&\frac{1}{2}\langle\nabla_e[x,y],e\rangle=\frac{u_1}{4u_i}\langle[e,[x,y]],e\rangle=\frac{u_1}{4}B([e,[x,y]],e)\\
&=&\frac{u_1}{4}B([e,e],[x,y])=0;
\end{eqnarray*}
\begin{eqnarray*}
\langle\nabla_x\nabla_ey,e\rangle&=&\frac{(2u_i-u_1)u_1}{4u_i^2}\langle[x,[e,y]],e\rangle=\frac{(2u_i-u_1)u_1}{4u_i}B([x,[e,y]],e)\\
&=&-\frac{(2u_i-u_1)u_1}{4u_i}B(\mathrm{ad}^2e(x),y)=-\frac{2u_i-u_1}{4u_i}\langle\mathrm{ad}^2e(x),y\rangle;
\end{eqnarray*}
\begin{eqnarray*}
\langle\nabla_{[e,x]}y,e\rangle&=&\frac{u_1}{2u_i}\langle[[e,x],y],e\rangle=\frac{u_1}{2}B([[e,x],y],e)=\frac{u_1}{2}B(\mathrm{ad}^2e(x),y)\\
&=&\frac{1}{2}\langle\mathrm{ad}^2e(x),y\rangle.
\end{eqnarray*}
It follows that
$\mathrm{Ric}(x,y)=\sum_{j=0}^3-\frac{u_1}{4u_{j+1}^2}\langle\sum_{k={i_j+1}}^{{i_{j+1}}}\mathrm{ad}^2e_k(x),y\rangle$.

Assume that $x,y\in{\mathfrak h}_2$. For any $e\in {\mathfrak h}_i$,
we have $\langle\nabla_e\nabla_xy,e\rangle=0$. If $e\in {\mathfrak
h}_1$,
\begin{eqnarray*}
\langle\nabla_x\nabla_ey,e\rangle+\langle\nabla_{[e,x]}y,e\rangle&=&\frac{2u_2-u_1}{4u_2}\langle[x,[e,y]],e\rangle+\frac{1}{2}\langle[[e,x],y],e\rangle\\
              &=&\frac{u_1^2}{4u_2^2}\langle\mathrm{ad}^2e(x),y\rangle;
\end{eqnarray*}
If $e\in {\mathfrak h}_2$, then we have
\begin{eqnarray*}
\langle\nabla_x\nabla_ey,e\rangle+\langle\nabla_{[e,x]}y,e\rangle &=& \frac{u_1}{4u_2}\langle[x,[e,y]],e\rangle+\frac{2u_2-u_1}{2u_2}\langle[[e,x],y],e\rangle\\
              &=&\frac{4u_2-3u_1}{4u_2}\langle\mathrm{ad}^2e(x),y\rangle;
\end{eqnarray*}
If $e\in {\mathfrak h}_3$, then we have
\begin{eqnarray*}
&&\langle\nabla_x\nabla_ey,e\rangle+\langle\nabla_{[e,x]}y,e\rangle \\
&=&\frac{(u_2+u_4-u_3)(u_3+u_4-u_2)}{4u_3u_4}\langle[x,[e,y]],e\rangle+\frac{u_2+u_3-u_4}{2u_3}\langle[[e,x],y],e\rangle\\
              &=&\frac{2u_4(u_2+u_3-u_4)-(u_2+u_4-u_3)(u_3+u_4-u_2)}{4u_2u_4}\langle\mathrm{ad}^2e(x),y\rangle;
\end{eqnarray*}
Similarly, if $e\in {\mathfrak h}_4$, we have
\begin{eqnarray*}&&\langle\nabla_x\nabla_ey,e\rangle+\langle\nabla_{[e,x]}y,e\rangle\\
&=&\frac{2u_3(u_2+u_4-u_3)-(u_2+u_3-u_4)(u_3+u_4-u_2)}{4u_2u_3}\langle\mathrm{ad}^2e(x),y\rangle.\end{eqnarray*}
Then the formula for $\mathrm{Ric}(x,y)$ follows when $x,y\in
{\mathfrak h}_2$, and the others are similar.
\end{proof}

\section{The compact simple Lie group $G$ is $F_4$ if ${\mathfrak h}_1$ is
simple and the involutions are inner-automorphisms} Firstly, we
recall some theorems on the study on involutions. Cartan and Gantmacher made great attributions on this field, and the theory on
the extension of involutions can be found in \cite{Be2}, which is
different in the method from that in \cite{Ya1}.

Let $G$ be a compact simple Lie group with the Lie algebra
${\mathfrak g}$ and $\theta$ be an involution of $G$. Then we have a
decomposition,
$${\mathfrak g}={\mathfrak k}+{\mathfrak p},$$ where ${\mathfrak
k}=\{x\in {\mathfrak g}|\theta(x)=x\}$ and ${\mathfrak p}=\{x\in
{\mathfrak g}|\theta(x)=-x\}$. Let ${\mathfrak t}$ be a Cartan
subalgebra of ${\mathfrak g}$ containing a Cartan subalgebra of
${\mathfrak k}$. 

From now on, we assume that the involutions are inner-automorphisms.

\begin{theorem}[Gantmacher Theorem]\label{chooseH} Let the notations be as above.
Then $\theta$ is conjugate with $e^{\mathrm{ad}H}$ under
$\mathrm{Aut} {\mathfrak g}$, where $H\in {\mathfrak t}$.
\end{theorem}

Let $\Pi=\{\alpha_1,\cdots,\alpha_n\}$ be a fundamental system of
${\mathfrak g}$ and $\phi=\sum_{i=1}^nm_i\alpha_i$ be the maximal
root respectively. Then we have
\begin{lemma}[\cite{Ya1}]\label{choosealpha}If $\theta\not=Id$, then there exists an
element $H$ satisfying 
\begin{equation}\label{H}
\langle H,\alpha_i\rangle=\pi\sqrt{-1};\quad \langle H,\alpha_j\rangle=0,
\forall j\not=i\end{equation} for some $i$. Here $m_i=1$ or $m_i=2$.
\end{lemma}

\begin{lemma}[\cite{Ya1}]\label{subalgebra}
Assume that $\alpha_i$ satisfies the identity~(\ref{H}).
\begin{enumerate}
  \item If $m_i=1$, then $\Pi-\{\alpha_i\}$ is the fundamental
  system of $\mathfrak k$, and $\phi$ and $-\alpha_i$ are the
  highest weights of $\mathrm{ad}_{\mathfrak p}{\mathfrak k}$.
  \item If $m_i=2$, then $\Pi-\{\alpha_i\}\cup\{-\phi\}$ is the fundamental
  system of $\mathfrak k$, and $-\alpha_i$ is the
  highest weight of $\mathrm{ad}_{\mathfrak p}{\mathfrak k}$.
\end{enumerate}
\end{lemma}

Assume that $\tau$ is an involution of $\mathfrak g$ satisfying
$\tau\theta=\theta\tau$. Then $\tau_{\mathfrak k}=\tau|_{\mathfrak k}$ is an involution
of ${\mathfrak k}$. That is to say, we can obtain every involution
$\tau$ of $\mathfrak g$ satisfying $\tau\theta=\tau\rho$ by
extending an involution of $\mathfrak k$ to an involution of
$\mathfrak g$. Since $\tau$ is an inner-automorphism of $\mathfrak
g$, we know that $\tau_{\mathfrak k}$ is an inner-automorphism.
Furthermore,
\begin{lemma}[\cite{Ya1}]
Let $\tau_{\mathfrak k}$ be an involution of ${\mathfrak k}$ which
is an inner-automorphism. Then the natural extension $\tau$ of
$\tau_{\mathfrak k}$ from ${\mathfrak k}$ to ${\mathfrak g}$ is an
inner-automorphism of ${\mathfrak g}$, and $\tau^2=Id$ or
$\tau^2=\theta$.
\end{lemma}

By Lemma~\ref{subalgebra}, ${\mathfrak k}={\mathfrak
z}\oplus{\mathfrak k}_1\oplus\cdots\oplus{\mathfrak k}_s$, where
${\mathfrak z}$ is the 1-dimensional center of $\mathfrak k$ or
${\mathfrak z}=0$, and every ${\mathfrak k}_i$ is a simple ideal of
$\mathfrak k$. If $\tau_{\mathfrak k}$ is the restrict of $\tau$
from $\mathfrak g$ to ${\mathfrak k}$, we have that every
$\tau_{\mathfrak k}^i=\tau_{\mathfrak k}|_{{\mathfrak k}_i}$ is both
an involution and an inner-automorphism of ${\mathfrak k}_i$. By
Theorems~\ref{chooseH}, for every $i$, there exists an element
$H_i\in {\mathfrak k}_i\cap{\mathfrak t}$ such that $\tau_{\mathfrak
k}^i$ is conjugate with $e^{\mathrm{ad}H_i}$ under $\mathrm{Aut}
{\mathfrak k}_i$. Here $H_i=0$ if $\tau_{\mathfrak k}^i=Id$. If
$H_i\not=0$, by Lemma~\ref{choosealpha}, we can choose $H_i$ such
that there exists a simple root $\alpha_{j_i}$ of ${\mathfrak k}_i$
satisfying the identity~(\ref{H}), here $\alpha_{j_i}$ lies in
$\Pi\cup\{-\phi\}$, and the coefficient of $\alpha_{j_i}$ in the
maximal root of ${\mathfrak k}_i$ is 1 or 2. Then $\tau_{\mathfrak
k}$ is conjugate with
\begin{equation}\label{K}e^{\mathrm{ad}(H_1+\cdots+H_s)}\end{equation}
under $\mathrm{Aut} {\mathfrak k}$. Furthermore,
\begin{lemma}[\cite{Ya1}] Let notations be as above.
\begin{enumerate}
  \item If the center of $\mathfrak k$ is 1-dimensional, then the extension of every
  involution of $\mathfrak k$ with the form~(\ref{K}) is an involution
  of $\mathfrak g$.
  \item Assume that $\mathfrak k$ has no center, $H_1,\cdots,H_p$
  are the non-zero elements in $H_1,\cdots,H_s$, and
  $m_{j_k}=1$ if $\alpha_{j_k}=-\phi$. Then the extension of every
  involution of $\mathfrak k$ with the form~(\ref{K}) is an involution
  of $\mathfrak g$ if and only if $\sum_{i=1}^pm_{j_i}$ is even.
  Here $m_{j_i}$ means the coefficient of $\alpha_{j_i}$ in $\phi$.
\end{enumerate}
\end{lemma}

\begin{remark} Clearly $\theta$ is invariant under $\mathrm{Aut} {\mathfrak
k}$ by the definition of ${\mathfrak k}$. Then by the above theory,
we can get every involution pair $(\tau, \theta)$ satisfying
$\tau\theta=\theta\tau$ in the sense of $\mathrm{Aut} {\mathfrak
g}$-conjugation.
\end{remark}

\begin{theorem}\label{thm}
Let $G$ be a compact simple Lie group with the Lie algebra
$\mathfrak g$ and $\tau,\theta$ be involutions of $\mathfrak g$
satisfying $\tau\theta=\theta\tau$. If the Lie subalgebra
${\mathfrak h}_1=\{x\in {\mathfrak g}|\theta(x)=x \text{ and }
\tau(x)=x\}$ is simple, then ${\mathfrak g}=F_4$.
\end{theorem}
\begin{proof}
Assume that ${\mathfrak k}=\{x\in {\mathfrak g}|\theta(x)=x\}$. By
the above discussions, $\tau$ is the extension of some involution of
$\mathfrak k$. By Theorem~\ref{chooseH}, $\theta$ is conjugate with
$e^{\mathrm{ad}H}$. By Lemma~\ref{choosealpha}, we can assume that
there is a simple $\alpha_i$ satisfying (\ref{H}). Assume that
$m_i=1$. Then by Lemma~\ref{subalgebra}, the center of $\mathfrak
k$ is 1-dimensional. It follows that ${\mathfrak h}_1$ isn't simple.
So $m_i=2$, and $\mathfrak k$ is simple if ${\mathfrak k}_1$ is
simple. Since $\mathfrak k$ is simple, by Lemma~\ref{subalgebra},
we have that
$${\mathfrak g}=B_l,{\mathfrak k}=D_l;\quad \text{ or }{\mathfrak
g}=E_7,{\mathfrak k}=A_7;\quad \text{ or } {\mathfrak
g}=E_8,{\mathfrak k}=D_8;\quad \text{ or } {\mathfrak
g}=F_4,{\mathfrak k}=B_4.$$ Since ${\mathfrak h}_1=\{x\in {\mathfrak
k}|\tau(x)=x\}$ is simple, we have that ${\mathfrak g}=F_4,
{\mathfrak k}=B_4, {\mathfrak h}_1=D_4$.
\end{proof}
We can choose a fundamental system
$\{\alpha_1,\alpha_2,\alpha_3,\alpha_4\}$ of $F_4$ such that the
involution pair in Theorem~\ref{thm} is described as follows. The
Dynkin diagram of $F_4$ corresponding to the fundamental system is
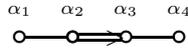
\begin{center}
\setlength{\unitlength}{0.7mm}
\begin{picture}(30,15)(0,-5)
\thicklines \multiput(0,0)(10,0){4}{\circle{2}}
\put(10,0){\circle{2}} \multiput(1,0)(10,0){1}{\line(1,0){8}}
\multiput(11,-0.5)(0,1){2}{\line(1,0){8}} \put(21,0){\line(1,0){8}}
\put(16,-1.5){$>$} \put(0,5){\makebox(0,0){$\scriptstyle \alpha_1$}}
\put(10,5){\makebox(0,0){$\scriptstyle \alpha_2$}}
\put(20,5){\makebox(0,0){$\scriptstyle \alpha_3$}}
\put(30,5){\makebox(0,0){$\scriptstyle \alpha_4$}}
\end{picture}
\end{center}
In the sense of conjugacy under $\mathrm{Aut} F_4$, the involution
$\theta$ is $e^{\mathrm{ad}H}$ such that
$$\langle H,\alpha_4\rangle=\pi\sqrt{-1};\quad \langle H,\alpha_j\rangle=0,
\forall j\not=4.$$ Let
$\phi=2\alpha_1+3\alpha_2+4\alpha_3+2\alpha_4$. Then the Dynkin
diagram of ${\mathfrak k}$ is
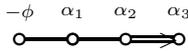
\begin{center}
\setlength{\unitlength}{0.7mm}
\begin{picture}(30,15)(0,-5)
\thicklines \multiput(0,0)(10,0){4}{\circle{2}}
\put(10,0){\circle{2}} \multiput(1,0)(10,0){2}{\line(1,0){8}}
\multiput(21,-0.5)(0,1){2}{\line(1,0){8}} \put(26,-1.5){$>$}
\put(0,5){\makebox(0,0){$\scriptstyle -\phi$}}
\put(10,5){\makebox(0,0){$\scriptstyle \alpha_1$}}
\put(20,5){\makebox(0,0){$\scriptstyle \alpha_2$}}
\put(30,5){\makebox(0,0){$\scriptstyle \alpha_3$}}
\end{picture}
\end{center}
The involution $\tau$ is the extension of $\tau_{\mathfrak k}$,
where $\tau_{\mathfrak k}=e^{\mathrm{ad}H_1}$ satisfies
$$\langle H_1,\alpha_3\rangle=\pi\sqrt{-1};\quad \langle H,\alpha_1\rangle=\langle H_1,\alpha_2\rangle=\langle
H_1,-\phi\rangle=0.$$ Then $\phi_1=-(\alpha_2+2\alpha_3+2\alpha_4)$
be the maximal root of $\mathfrak k$. Then the Dynkin diagram of
${\mathfrak h}_1$ is
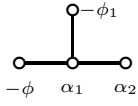
\begin{center}
\setlength{\unitlength}{0.7mm}
\begin{picture}(30,20)(0,-5) \thicklines
\multiput(0,0)(10,0){3}{\circle{2}} \put(10,10){\circle{2}}
\multiput(1,0)(20,0){1}{\line(1,0){8}}
\multiput(11,0)(2,0){1}{\line(1,0){8}} \put(10,1){\line(0,1){8}}
\put(0,-5){\makebox(0,0){$\scriptstyle -\phi$}}
\put(10,-5){\makebox(0,0){$\scriptstyle \alpha_1$}}
\put(15,10){\makebox(0,0){$\scriptstyle -\phi_1$}}
\put(20,-5){\makebox(0,0){$\scriptstyle \alpha_2$}}
\end{picture}
\end{center}
Let ${\mathfrak g}={\mathfrak h_1}\oplus{\mathfrak
h_2}\oplus{\mathfrak h_3}\oplus{\mathfrak h_4}$, where
\begin{eqnarray*}
  && {\mathfrak h_1}=\{x\in {\mathfrak
g}|\theta(x)=x, \tau(x)=x\}, \quad {\mathfrak h_2}=\{x\in {\mathfrak
g}|\theta(x)=x, \tau(x)=-x\},\\
  &&{\mathfrak h_3}=\{x\in {\mathfrak
g}|\theta(x)=-x, \tau(x)=x\},\quad {\mathfrak h_4}=\{x\in {\mathfrak
g}|\theta(x)=-x, \tau(x)=-x\}.
\end{eqnarray*}
Clearly $\mathfrak k={\mathfrak h}_1\oplus{\mathfrak h_2}$, which is
a decomposition of $\mathfrak k$ corresponding to the involution
$\tau_{\mathfrak k}$. By Lemma~\ref{subalgebra}, ${\mathfrak h}_2$
is the irreducible representation of ${\mathfrak h}_1$ with the
highest weight $-\alpha_3$, which is a fundamental dominant weight corresponding
to $\alpha_2$.

It is easy to see that ${\mathfrak k}^{1}=\{x\in {\mathfrak
g}|\theta\tau(x)=x\}={\mathfrak h}_1\oplus{\mathfrak h}_4$, which is
a decomposition of ${\mathfrak k}^1$ corresponding to the involution
$\tau_{{\mathfrak k}^1}$. Clearly
$\{\alpha_1,\alpha_2,\alpha_3'=\alpha_3+\alpha_4,\alpha_4'=-\alpha_4\}$
is another fundamental system of $F_4$. The Dynkin diagram
corresponding to the fundamental system is
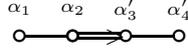
\begin{center}
\setlength{\unitlength}{0.7mm}
\begin{picture}(30,15)(0,-5)
\thicklines \multiput(0,0)(10,0){4}{\circle{2}}
\put(10,0){\circle{2}} \multiput(1,0)(10,0){1}{\line(1,0){8}}
\multiput(11,-0.5)(0,1){2}{\line(1,0){8}} \put(21,0){\line(1,0){8}}
\put(16,-1.5){$>$} \put(0,5){\makebox(0,0){$\scriptstyle \alpha_1$}}
\put(10,5){\makebox(0,0){$\scriptstyle \alpha_2$}}
\put(20,5){\makebox(0,0){$\scriptstyle \alpha_3'$}}
\put(30,5){\makebox(0,0){$\scriptstyle \alpha_4'$}}
\end{picture}
\end{center}
The involution $\theta\tau=e^{\mathrm{ad}(H+H_1)}$ satisfies
$$\langle H+H_1,\alpha_4'\rangle=\pi\sqrt{-1};\quad \langle H+H_1,\alpha_1\rangle=\langle H+H_1,\alpha_2\rangle=\langle H+H_1,\alpha_3'\rangle=0.$$
And
$\phi=2\alpha_1+3\alpha_2+4\alpha_3'+2\alpha_4'=2\alpha_1+3\alpha_2+4\alpha_3+2\alpha_4$,
the Dynkin diagram of ${\mathfrak k}^1$ is
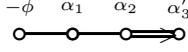
\begin{center}
\setlength{\unitlength}{0.7mm}
\begin{picture}(30,15)(0,-5)
\thicklines \multiput(0,0)(10,0){4}{\circle{2}}
\put(10,0){\circle{2}} \multiput(1,0)(10,0){2}{\line(1,0){8}}
\multiput(21,-0.5)(0,1){2}{\line(1,0){8}} \put(26,-1.5){$>$}
\put(0,5){\makebox(0,0){$\scriptstyle -\phi$}}
\put(10,5){\makebox(0,0){$\scriptstyle \alpha_1$}}
\put(20,5){\makebox(0,0){$\scriptstyle \alpha_2$}}
\put(30,5){\makebox(0,0){$\scriptstyle \alpha_3'$}}
\end{picture}
\end{center}
The involution $\tau=e^{\mathrm{ad}H_1}$ restricted on ${\mathfrak
k}^1$ satisfies
$$\langle H_1,\alpha_3'\rangle=\pi\sqrt{-1};\quad \langle H,\alpha_1\rangle=\langle H_1,\alpha_2\rangle=\langle
H_1,-\phi\rangle=0.$$ Then $\phi_1=-(\alpha_2+2\alpha_3)$ is the maximal
root of $\mathfrak k$, and the Dynkin diagram of ${\mathfrak h}_1$
is
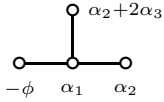
\begin{center}
\setlength{\unitlength}{0.7mm}
\begin{picture}(30,20)(0,-5) \thicklines
\multiput(0,0)(10,0){3}{\circle{2}} \put(10,10){\circle{2}}
\multiput(1,0)(20,0){1}{\line(1,0){8}}
\multiput(11,0)(2,0){1}{\line(1,0){8}} \put(10,1){\line(0,1){8}}
\put(0,-5){\makebox(0,0){$\scriptstyle -\phi$}}
\put(10,-5){\makebox(0,0){$\scriptstyle \alpha_1$}}
\put(20,10){\makebox(0,0){$\scriptstyle \alpha_2+2\alpha_3$}}
\put(20,-5){\makebox(0,0){$\scriptstyle \alpha_2$}}
\end{picture}
\end{center}
By Lemma~\ref{subalgebra}, ${\mathfrak h}_2$ is the irreducible
representation of ${\mathfrak h}_1$ with the highest weight
$-\alpha_3'$. It is easy to see that
$\{-\phi,\alpha_1,\alpha_2,-\phi_1\}$ is
another fundamental system of ${\mathfrak h}_1$. For this fundamental system, ${\mathfrak h}_4$ is the irreducible representation
of ${\mathfrak h}_1$ with the highest weight $\alpha_1+2\alpha_2+3\alpha_3+\alpha_4$, which is a
fundamental dominant weight corresponding to $-\phi$. Similarly, we have 
\begin{proposition}\label{thm1}
For the fundamental system $\{-\phi,\alpha_1,\alpha_2,-\phi_1\}$ of ${\mathfrak h}_1$, the highest weights of ${\mathfrak h}_2$, ${\mathfrak h}_3$ and ${\mathfrak h}_4$ as $\mathrm{ad} ({\mathfrak h}_1)$-
modules are fundamental dominant weights of ${\mathfrak h}_1$ corresponding to $\alpha_2$, $-\phi_1$ and $-\phi$ respectively.
\end{proposition}
\section{Einstein metrics on the compact simple Lie group $F_4$}
Let ${\mathfrak g}$ be a compact simple Lie algebra, $\rho:
{\mathfrak g}\rightarrow {\mathfrak gl}(V)$ be an irreducible
representation of ${\mathfrak g}$, $(\ ,\ )$ be a non-degenerate,
associative and symmetric bilinear form on $\mathfrak g$. Assume
that $\{e_1,\cdots,e_n\}$ is an orth-normal basis of ${\mathfrak g}$
corresponding to $(\ ,\ )$. Then for any $x\in {\mathfrak g}$,
$$[\sum_i\rho^2(e_i),\rho(x)]=0.$$
Since $\rho$ is irreducible, Schur Lemma implies that
$\sum_i\rho^2(e_i)$ is a constant. Furthermore assume that the highest weight of
$\rho$ is $\lambda$. Then
\begin{lemma}\label{lemma}
Let notations be as above. Then
$$\sum_i\phi^2(e_i)=((\lambda+\delta,\lambda+\delta)-(\delta,\delta))Id,$$
where $\delta$ denotes the half of the sum of positive roots of
$\mathfrak g$.
\end{lemma}

Let ${\mathfrak g}$, $\theta$, ${\mathfrak k}$ and ${\mathfrak k}_1$
be as that in the previous section. Assume that $B$ is the associative inner product on
$\mathfrak g$ such that $B(\alpha_1,\alpha_1)=2$. Let
$\{e_1,\cdots,e_n\}$ be an orthogonal basis of ${\mathfrak g}$ with
respect to $B$, where $\{e_1,\cdots,e_{i_1}\}$ is the basis of
${\mathfrak h}_1$, $\{e_{i_1}+1,\cdots,e_{i_2}\}$ is the basis of
${\mathfrak h}_2$, $\{e_{i_2}+1,\cdots,e_{i_3}\}$ is the basis of
${\mathfrak h}_3$ and $\{e_{i_3}+1,\cdots,e_{n}\}$ is the basis of
${\mathfrak h}_4$. Since ${\mathfrak p}$ is an irreducible
representation of ${\mathfrak k}$ with the highest weight
$-\alpha_4$, by Lemma~\ref{lemma}, we have:
$$\sum_{j=1}^{i_2}\mathrm{ad}e_j^2|_{\mathfrak p}=((-\alpha_4,-\alpha_4)+2(\delta_{\mathfrak
k},-\alpha_4))Id=9Id.$$ Since ${\mathfrak k}$ is simple, we have
$$\sum_{j=1}^{i_2}\mathrm{ad}e_j^2|_{\mathfrak k}=((\phi_{\mathfrak k},\phi_{\mathfrak k})+2(\delta_{\mathfrak
k},\phi_{\mathfrak k}))Id=14Id.$$ Since ${\mathfrak k}$ and
${\mathfrak h}_1$ are simple, similarly, we have
$$ \sum_{j=1}^{i_1}\mathrm{ad}e_j^2|_{\mathfrak h_1}=12Id; \quad  \sum_{j=1}^{i_1}\mathrm{ad}e_j^2|_{\mathfrak
h_2}=7Id.$$
It follows that
\begin{eqnarray*}
&&\sum_{j=i_1+1}^{i_2}\mathrm{ad}e_j^2|_{\mathfrak h_1}=\sum_{j=1}^{i_2}\mathrm{ad}e_j^2|_{\mathfrak k}-\sum_{j=1}^{i_1}\mathrm{ad}e_j^2|_{\mathfrak h_1}=2Id;\\
&&\sum_{j=i_1+1}^{i_2}\mathrm{ad}e_j^2|_{\mathfrak h_2}=\sum_{j=1}^{i_2}\mathrm{ad}e_j^2|_{\mathfrak k}-\sum_{j=1}^{i_1}\mathrm{ad}e_j^2|_{\mathfrak h_2}=7Id.
\end{eqnarray*}
Similarly, by Proposition~\ref{thm1}, we have

\begin{eqnarray*}
&&\sum_{j=1}^{i_1}\mathrm{ad}e_j^2|_{\mathfrak h_1}=12Id, \sum_{j=1}^{i_1}\mathrm{ad}e_j^2|_{\mathfrak h_2}=\sum_{j=1}^{i_1}\mathrm{ad}e_j^2|_{\mathfrak h_3}=\sum_{j=1}^{i_1}\mathrm{ad}e_j^2|_{\mathfrak h_4}=7Id, \\
&&\sum_{j=i_1+1}^{i_2}\mathrm{ad}e_j^2|_{\mathfrak h_1}=\sum_{j=i_1+1}^{i_2}\mathrm{ad}e_j^2|_{\mathfrak h_3}=\sum_{j=i_1+1}^{i_2}\mathrm{ad}e_j^2|_{\mathfrak h_4}=2Id, \sum_{j=i_1+1}^{i_2}\mathrm{ad}e_j^2|_{\mathfrak h_2}=7Id,\\
&& \sum_{j=i_2+1}^{i_3}\mathrm{ad}e_j^2|_{\mathfrak h_1}=\sum_{j=i_2+1}^{i_4}\mathrm{ad}e_j^2|_{\mathfrak h_2}=\sum_{j=i_2+1}^{i_3}\mathrm{ad}e_j^2|_{\mathfrak h_4}=2Id, \sum_{j=i_2+1}^{i_3}\mathrm{ad}e_j^2|_{\mathfrak h_3}=7Id, \\
&&\sum_{j=i_3+1}^{i_4}\mathrm{ad}e_j^2|_{\mathfrak h_1}=\sum_{j=i_3+1}^{i_4}\mathrm{ad}e_j^2|_{\mathfrak h_3}=\sum_{j=i_3+1}^{i_4}\mathrm{ad}e_j^2|_{\mathfrak h_4}=2Id,\sum_{j=i_3+1}^{i_4}\mathrm{ad}e_j^2|_{\mathfrak h_2}=7Id.
\end{eqnarray*}

By Propositions~\ref{Ricci} and~\ref{thm1}, we have
\begin{proposition}\label{F4}
Let ${\mathfrak g}=F_4$, $\theta$, $\tau$, $B$ be as above. Then the left invariant
metric which is $\mathrm{Ad}({H_1})$ is of the form (\ref{involutionpair}). Furthermore the
metric $\langle\ ,\ \rangle$ with the form~(\ref{involutionpair}) is
Einstein if and only if, for some $u_1, u_2, u_3, u_4\in {\mathbb
R}^+$,
\begin{eqnarray}
  && -\frac{3}{u_1}-\frac{u_1}{2u_2^2}-\frac{u_1}{2u_3^2}-\frac{u_1}{2u_4^2}  \label{1}\\
  &=& \frac{7u_1}{2u_2^2}-\frac{9}{u_2}+\frac{u_3^2+u_4^2-u_2^2}{u_2u_3u_4}  \label{2}\\
  &=& \frac{7u_1}{2u_3^2}-\frac{9}{u_3}+\frac{u_2^2+u_4^2-u_3^2}{u_2u_3u_4}  \label{3}\\
  &=& \frac{7u_1}{2u_4^2}-\frac{9}{u_4}+\frac{u_2^2+u_3^2-u_4^2}{u_2u_3u_4}. \label{4}
\end{eqnarray}
\end{proposition}

The following is to give the solution of the above equations.
Firstly, assume that $u_2=u_3=1$. Then $(\ref{2})=(\ref{3})$. By
$(\ref{1})=(\ref{2})$ and $(\ref{2})=(\ref{4})$, we have
\[ \left\{ \begin{aligned}
&9u_1^2u_4^2-18u_1u_4^2+2u_1u_4^3+6u_4^2+u_1^2=0,\\
&7u_1(u_4^2-1)+2u_4(2u_4-7)(u_4-1)=0.
\end{aligned}\right.\]
If $u_4=1$, then $u_1=1$ or $u_1=\frac{3}{5}$. If $u_4\not=1$, by
the second equation, we have $$u_1=\frac{2u_4(7-2u_4)}{7u_4+7}.$$
Putting into the first equation, we have
$44u_4^4-182u_4^3+505u_4^2-644u_4+245=0.$ That is,
$$(11u_4-7)(4u_4^3-14u_4^2+37u_4-35)=0.$$ Then $u_4=\frac{7}{11}$ or
$u_4\approx1.3842$. It follows that $u_1=\frac{7}{11}$ or
$u_1\approx 0.7019$ respectively.

Secondly, assume that $u_1=u_2=1$. By $(\ref{1})=(\ref{2})$, we have
$$(u_3-u_4)^2(2u_3u_4+1)=0.$$ Since $u_3u_4>0$, we have that $u_3=u_4$. Then
$(\ref{3})=(\ref{4})$. By $(\ref{2})=(\ref{3})$, we have
$7u_3^2-18u_3+11=0$. It follows that $u_3=1$ or $u_3=\frac{11}{7}$.

Finally, we can assume that $u_i\not=u_j$ if $i\not=j$ by
Theorem~\ref{thm1}. Assume that $u_1=1$. By $(\ref{2})=(\ref{3})$
and $(\ref{2})=(\ref{4})$, we have
\[ \left\{ \begin{aligned}
    &   \frac{7}{2}(u_2+u_3)u_4-9u_2u_3u_4+2(u_2+u_3)u_2u_3=0; \\
    &   \frac{7}{2}(u_2+u_4)u_3-9u_2u_3u_4+2(u_2+u_4)u_2u_4=0.
\end{aligned}\right.\]
Subtracting, we have $u_2+u_3+u_4=\frac{7}{4}.$ Then
$u_2u_3=\frac{7u_4(\frac{7}{4}-u_4)}{22u_4-7}$. Together with
$(\ref{1})=(\ref{2})$, there is no nonzero and real number solution.

Thus we have four solutions:
\begin{enumerate}
   \item $(u_1,u_2,u_3,u_4)=(1,1,1,1);$
   \item $(u_1,u_2,u_3,u_4)=(\frac{7}{11},1,1,\frac{7}{11});$
   \item $(u_1,u_2,u_3,u_4)=(\frac{3}{5},1,1,1);$
   \item $(u_1,u_2,u_3,u_4)\approx(0.7019,1,1,1.3842).$
\end{enumerate}

By Lemma~\ref{natural2}, the first three correspond to naturally
reductive Einstein metrics, and the fourth one corresponds to a
non-naturally reductive Einstein metric which implies Theorem~\ref{main}. For the naturally reductive
Einstein metric, the first one is bi-invariant, the second one is
also given by \cite{DZ1,Je}, and the third one is new.

\section{Acknowledgments}
This work is supported by NSFC (No. 11001133). The first author would like to thank Prof.
J.A. Wolf for the helpful conversation and suggestions.


\begin{thebibliography}{99}
\bibitem{AK}
\textsc{D. Alekseevsky} and \textsc{B. Kimel'fel'd}, Structure of homogeneous Riemannian spaces with zero Ricci curvature, \textit{Functional Anal. Appl.}
\textbf{9} (1975), 97--102.

\bibitem{AMY1}
\textsc{A. Arvanitoyeorgos}, \textsc{K. Mori} and \textsc{Y.
Sakane}, Einstein metrics on compact Lie groups which are not
naturally reducitive, \textit{Geom. Dedicata} \textbf{160} (2012),
261--285.

\bibitem{Be2}
\textsc{M. Berger}, Les espaces sym\'etriques noncompacts,
\textit{Ann. Sci. \'Ecole Norm. Sup. (3)} \textbf{74} (1957),
85--177.

\bibitem{Be1}
\textsc{A. Besse}, Einstein manifolds, \textit{Ergeb. Math.} \textbf{10} (1987), Springer-Verlag, Berlin-Heidelberg.

\bibitem{Bo1}
\textsc{C. B\"ohm}, Homogeneous Einstein metrics and simplicial
complexes, \textit{J. Differetial Geom.} \textbf{67} (2004),
79--165.

\bibitem{Bo2}
\textsc{C. B\"ohm} and \textsc{M.M. Kerr}, Low-dimensional
homogeneous Einstein manifolds, \textit{Trans. Amer. Math. Soc.}
\textbf{358(4)} (2006), 1455--1468.

\bibitem{Bo3}
\textsc{C. B\"ohm}, \textsc{M. Wang} and \textsc{W. Ziller}, A
variational approach for homogeneous Einstein metrics, \textit{Geom.
Funct. Anal.} \textbf{14} (2004), 681--733.

\bibitem{CL}
\textsc{Z.Q. Chen} and \textsc{K. Liang}, Classification of analytic
involution pairs of Lie groups (in Chinese), \textit{Chinese Ann.
Math. Ser. A} \textbf{26} (2005), no. 5, 695--708; translation in
\textit{Chinese J. Contemp. Math.} \textbf{26} (2005), no. 4,
411--424 (2006).

\bibitem{CH}
\textsc{M.K. Chuan} and \textsc{J.S. Huang}, Double Vogan diagrams and semisimple symmetric
spaces, \textit{Trans. Amer. Math. Soc.} \textbf{362} (2010),
1721--1750.

\bibitem{DZ1}
\textsc{J.E. D'Atri} and \textsc{W. Ziller}, Naturally reductive
metrics and Einstein metrics on compact Lie groups, \textit{Memoirs
of Amer. Math. Soc.} \textbf{215} (1979).

\bibitem{GLP}
\textsc{G.W. Gibbons}, \textsc{H. L\"u} and \textsc{C.N. Pope},
Einstein metrics on group manifolds and cosets, \textit{J. Geom.
Phys.} \textbf{61} (2011), no. 5, 947--960.

\bibitem{He1}
\textsc{J. Heber}, Noncompact homogeneous Einstein spaces, \textit{Invent. math.} \textbf{133} (1998), 279--352.

\bibitem{Je}
\textsc{G.R. Jensen}, Einstein metrics on principal fibre bundles,
\textit{J. Differential Geom.} \textbf{8} (1973), 599--614.

\bibitem{Ja1}
\textsc{J. Lauret}, Einstein solvmanifolds are standard,
\textit{Ann. of Math. (2)} \textbf{172} (2010), no. 3, 1859--1877.

\bibitem{Mo1}
\textsc{K. Mori}, Left invariant Einstein metrics on $SU(n)$ that
are not naturally reductive, Master Thesis (in Japanese) Osaka
University 1994, English translation Osaka University RPM 96-10
(preprint series) 1996.

\bibitem{Mu1}
\textsc{A.H. Mujtaba}, Homogeneous Einstein metrics on $SU(n)$,
\textit{J. Geom. Phys.} \textbf{62} (2012), no. 5, 976--980.

\bibitem{NR1}
\textsc{Yu.G. Nikonorov}, \textsc{E.D. Rodionov} and \textsc{V.V.
Slavskii}, Geometry of homogeneous Riemannian manifolds, \textit{J.
Math. Sci.} \textbf{146(6)} (2007), 6313--6390.

\bibitem{Po1}
\textsc{C.N. Pope}, Homogeneous Einstein metrics on $SO(n)$, arXiv:
1001.2776, 2010.

\bibitem{Sa1}
\textsc{A. Sagle}, Some homogeneous Einstein manifolds,
\textit{Nagoya Math. J.} \textbf{39} (1970), 81--106.

\bibitem{WZ1}
\textsc{M. Wang} and \textsc{W. Ziller}, Existence and non-existence
of homogeneous Einstein metrics, \textit{Invent. Math.} \textbf{84}
(1986), 177--194.

\bibitem{Ya1}
\textsc{Z.D. Yan}, Real semisimple Lie algebras (in chinese), Nankai
University Press, Tianjin, 1998.
\end{thebibliography}
\end{document}